\documentclass[11pt,letterpaper]{amsart}
\usepackage[utf8]{inputenc}
\usepackage{verbatim}
\usepackage{amsmath,amssymb,amsfonts,amsthm}
\usepackage{bigints}
\usepackage{hyperref}
\usepackage{array}

\usepackage{verbatim}
\usepackage{amsmath,amssymb,amsfonts,amsthm}
\usepackage{bigints}
\usepackage{hyperref}
\usepackage{array}

\usepackage{amsthm}

\usepackage{mathtools}

\mathtoolsset{showonlyrefs}

\newtheorem{definition}{Definition}[section]

\newtheorem{theorem}[definition]{Theorem}

\newtheorem{lemma}[definition]{Lemma}
\newtheorem{corollary}[definition]{Corollary}
\newtheorem{remark}[definition]{Remark}

\newcommand\R{\mathbb{R}}

\newcommand{\norm}[1]{ \left\lVert#1\right\rVert}
\newcommand{\mathbbm}[1]{\text{\usefont{U}{bbm}{m}{n}#1}}

\thanks{M. Saucedo is supported by  the Spanish Ministry of Universities through the FPU contract FPU21/04230.
 S. Tikhonov is supported
by PID2023-150984NB-I00, 2021 SGR 00087. This work is supported by
the CERCA Programme of the Generalitat de Catalunya and the Severo Ochoa and Mar\'ia de Maeztu
Program for Centers and Units of Excellence in R\&D (CEX2020-001084-M)}

\author{Miquel Saucedo}
\address{M.  Saucedo,  Centre de Recerca Matemàtica\\
Campus de Bellaterra, Edifici C
08193 Bellaterra (Barcelona), Spain}
\email{miquelsaucedo98@gmail.com }

\author{Sergey Tikhonov}
\address{S. Tikhonov, Centre de Recerca Matem\`{a}tica\\
Campus de Bellaterra, Edifici C
08193 Bellaterra (Barcelona), Spain;
ICREA, Pg. Lluís Companys 23, 08010 Barcelona, Spain,
 and Universitat Autònoma de Barcelona.}
\email{ stikhonov@crm.cat}


\title{Poisson summation formula in weighted Lebesgue spaces%. The end-point cases
}

\begin{document}
\begin{abstract}
    We characterize the parameters $(\alpha,\beta,p,q)$ for which the condition $f|x|^\alpha\in L^p$ and $\widehat{f}|\xi|^\beta\in L^q$ implies the validity of the Poisson summation formula, thus completing the study of Kahane and Lemarié-Rieusset.
\end{abstract}
\maketitle

%\tableofcontents
\section{Introduction}

The Fourier transform of an integrable function $f: \mathbb{R}\to \mathbb{C}$, 
is defined by
 %Define for $f\in L_1(\mathbb{R}^d)$ the Fourier transform
 \begin{eqnarray*} \widehat{f}(\xi)=\int_{\mathbb{R}} {f}(x) e^{-2\pi i  x \xi } d x.
 \end{eqnarray*}
The classical Poisson summation formula (PSF) states that if 
 \begin{eqnarray*}
   f|x|^\alpha, \widehat{f}|\xi|^\beta \in L^\infty, \quad \alpha,\beta>1,
\end{eqnarray*}
then $f$ and $\widehat{f}$ are both continuous and the formula

\begin{equation}\label{eqn:einstein}
    \lim_N P_N(f)=\lim_N P_N(\widehat{f})
    \tag{PSF}
  \end{equation}
is true,
 where $$P_N(g)=\sum_{|k|\leq N} g(k).$$
By a classical example of Katznelson it is well known that there exist continuous functions with $f,\widehat{f}\in L^1$
for which the Poisson summation formula fails, in the sense that both sides of \eqref{eqn:einstein}
 converge to different values.

In \cite{Kahane1994},
 Kahane and Lemarié-Rieusset obtained the following almost complete characterization of the 
 validity of \eqref{eqn:einstein} under the condition $f|x|^\alpha\in L^p$ and $\widehat{f}|\xi|^\beta\in L^q$:
 
 \begin{theorem} \cite{Kahane1994}
\label{theorem:kah} Let
$1\le p,q\le\infty$ and
$\alpha-\frac{1}{p'}, \beta- \frac{1}{q'}>0$. Let $f$ be a locally integrable function on $\mathbb{R}$ such that 
\begin{equation}
    \label{space}
\norm{f |x|^{\alpha}}_p + \norm{\widehat{f}|\xi|^{\beta}}_q< \infty.
\end{equation}
Then,
\begin{enumerate}
    \item if $(\alpha - \frac{1}{p'})(\beta- \frac{1}{q'})>\frac{1}{pq}$,  $\lim_N P_N(f)= \lim_N P_N(\widehat{f})$;

    \item 
    if $(\alpha - \frac{1}{p'})(\beta- \frac{1}{q'})<\frac{1}{pq}$, there exists a function  $f$ for which $P_N(f)$ and $P_N(\widehat{f})$ converge to different values.
\end{enumerate}
\end{theorem}

%??? Here, of course, the question is:
%Under which conditions on $u, v$, the condition
%$\norm{f u}_p + \norm{\widehat{f}v}_q< \infty
%$ implies
%$\lim_N P_N(f)= \lim_N P_N(\widehat{f})$.

%We can just follow the proof from \cite{Kahane1994}. If $u$ and $v$ satisfy growth conditions of type
%$u(sx)\lesssim u(x)s^\alpha$ it is easy. But we can do better. 

  Kahane and Lemarié-Rieusset also considered the question of the absolute convergence of the series $P_N(f)$ and obtained the following theorem:
\begin{theorem} \cite{Kahane1994}
\label{theorem:kah2} Let
$1\le p,q\le\infty$ and
$\alpha-\frac{1}{p'}, \beta- \frac{1}{q'}>0$. Let $f$ be a locally integrable function on $\mathbb{R}$ such that 
\eqref{space} holds.

Then, 
\begin{enumerate}
    \item if $(\alpha - \frac{1}{p'})(\beta- \frac{1}{q'})>\max(\frac{1}{pq},\frac{1}{2p}) $, $\lim_N P_N(|f|)<\infty$;

    \item if $q>2$ and
   $(\alpha - \frac{1}{p'})(\beta- \frac{1}{q'})<\frac{1}{2p}$,  there exists a function $f$ for which  $\lim_N P_N(|f|)=\infty$.
   
  %  \item if
  % $(\alpha - \frac{1}{p'})(\beta- \frac{1}{q'})<\max(\frac{1}{pq},\frac{1}{2p})$,  there exists a function $f$ for which  $\lim_N P_N(|f|)=\infty$.
\end{enumerate}
\end{theorem}

%NOTE that (2) was proved only for
%   $
%   \frac{1}{pq}<(\alpha - \frac{1}{p'})(\beta- \frac{1}{q'})<\max(\frac{1}{pq},\frac{1}{2p})$!

In this paper we complete the missing cases in Theorems \ref{theorem:kah} and \ref{theorem:kah2}.  
In particular, we show 
 %It is interesting to note that 
  in Theorem \ref{theorem:mainth} that for  $(\alpha - \frac{1}{p'})(\beta- \frac{1}{q'})=\frac{1}{pq}$,  $(p,q)\neq (1,1)$ if one side of \eqref{eqn:einstein}
 converges (to a finite or infinite value),  the other side must also converge to the same value. However, it is possible for neither side to converge.

See also \cite{benedetto} and \cite{Charly}
for various conditions for \eqref{eqn:einstein}
 to be true.

 \begin{comment}
     Finally, if all
$f$, $f'$,
$\widehat{f}$,
$\widehat{f}'$ integrable then
the Poisson
summation formula is valid [Bondarenko, Seip].
Later on, Gröchenig proved that the Poisson summation formula
holds for functions in the Feichtinger algebra.
 \end{comment}

We will use the notation $F \lesssim G$  to mean that $F \le C G$ with a constant $C=C(\alpha,\beta, \gamma, p,q)$ that may change from
line to line; $F \approx G$ means that both $F \lesssim G$ and $G \lesssim F$ hold. Besides, $ \norm{F}_{p}=
\norm{F}_{L^p(\R)}.
$

\section{Main results}
Our main results are the following extensions of Theorems \ref{theorem:kah} and 
\ref{theorem:kah2}:
\begin{theorem} 
\label{theorem:mainth} Let
$1\le p,q\le\infty$ and
$\alpha-\frac{1}{p'}, \beta- \frac{1}{q'}>0$. Let $f$ be a locally integrable function on $\mathbb{R}$ such that 
\eqref{space} holds.

Then,
\begin{enumerate}
    \item if $(\alpha - \frac{1}{p'})(\beta- \frac{1}{q'})>\frac{1}{pq}$, $\lim_N P_N(f)= \lim_N P_N(\widehat{f})$;
    \item
    if $(\alpha - \frac{1}{p'})(\beta- \frac{1}{q'})=\frac{1}{pq}$ and $(p,q)= (1,1)$, $\lim_N P_N(f)= \lim_N P_N(\widehat{f})$;
    \item
    if $(\alpha - \frac{1}{p'})(\beta- \frac{1}{q'})=\frac{1}{pq}$ and $(p,q)\neq (1,1)$,  neither $\lim_N P_N(\widehat{f})$ nor $\lim_N P_N({f})$ need to exist, but if one does, then  $\lim_N P_N(f)= \lim_N P_N(\widehat{f})$;
    %If $(p,q)=(1,1)$ we are in case (1).
    \item 
    if $(\alpha - \frac{1}{p'})(\beta- \frac{1}{q'})<\frac{1}{pq}$, there exists a function $f$ for which $P_N(f)$ and $P_N(\widehat{f})$ converge to different values.
\end{enumerate}
\end{theorem}

\begin{theorem} 
\label{theorem:extkah2} Let
$1\le p,q\le\infty$ and
$\alpha-\frac{1}{p'}, \beta- \frac{1}{q'}>0$. Let $f$ be a locally integrable function on $\mathbb{R}$ such that 
\eqref{space} holds.

Then,
\begin{enumerate}
    \item if $(\alpha - \frac{1}{p'})(\beta- \frac{1}{q'})>\max(\frac{1}{pq},\frac{1}{2p}) $,  $\lim_N P_N(|f|)<\infty$;
    \item
    if $(\alpha - \frac{1}{p'})(\beta- \frac{1}{q'})=\frac{1}{pq}$ and $(p,q)= (1,1)$, $\lim_N P_N(|f|)<\infty$;
    \item if
   $(\alpha - \frac{1}{p'})(\beta- \frac{1}{q'})\leq \max(\frac{1}{pq},\frac{1}{2p}) $ and $(p,q)\neq (1,1)$, there exists a function $f$ for which  $\lim_N P_N(|f|)=\infty$. 
\end{enumerate}
\end{theorem}
The following remarks are in order:
\begin{remark}

\label{remark}

\begin{enumerate}
\item The assumption $\alpha-\frac{1}{p'}, \beta- \frac{1}{q'}>0$ in \eqref{space} is natural because it implies that it guarantees that 
\begin{equation*}
    \norm{f}_1+ \norm{\widehat{f}\,}_1\lesssim     \norm{f |x|^{\alpha}}_p + \norm{\widehat{f}|\xi|^{\beta}}_q.
\end{equation*}Thus, both $f$ and $\widehat{f}$ can be defined pointwise. In particular, we have the equivalence 
\begin{equation}
    \norm{f |x|^{\alpha}}_p + \norm{\widehat{f}|\xi|^{\beta}}_q \approx \norm{f (1+|x|)^{\alpha}}_p + \norm{\widehat{f}(1+|\xi|)^{\beta}}_q.
\end{equation}
Observe that for $\varepsilon_1, \varepsilon_2\geq 0$,
\begin{equation}
    \norm{f |x|^{\alpha}}_p + \norm{\widehat{f}|\xi|^{\beta}}_q \lesssim 
    \norm{f |x|^{\alpha+\varepsilon_1}}_p + \norm{\widehat{f}|\xi|^{\beta+\varepsilon_2}}_q.
\end{equation}

\item In fact, by analyzing the proof in \cite{Kahane1994} we see that if either $\alpha-\frac{1}{p'}$ or $ \beta- \frac{1}{q'}$ is not strictly positive, then both sides of \eqref{eqn:einstein}
 can converge to different values even if we  assume that $f,\widehat{f}\in L ^ 1$.
%\eqref{space} holds
\item The space of functions $f$ which satisfy \eqref{space} is a Banach space, and the Schwartz functions are dense in it.
\item   In case (3) of Theorem \ref{theorem:mainth} we actually show that there exists  $\gamma>0$ depending on $(\alpha,\beta,p,q)$ such that $$\lim_{N\to\infty} (P_N(f)-P_{N^\gamma}(\widehat{f}\,))=0.$$ 

   \item In order that 
   $\lim_N P_N(|f|)+\lim_N P_N(|\widehat{f}|)
   <\infty$ for  
any locally integrable function    $f$ satisfying 
\eqref{space},
it is necessary and sufficient that 
$$(\alpha - \frac{1}{p'})(\beta- \frac{1}{q'})>\max(\frac{1}{pq},\frac{1}{2p},\frac{1}{2q}), \qquad (p,q)\ne (1,1)$$
and 
$$
\alpha \beta\ge 1,
%(\alpha - \frac{1}{p'})(\beta- \frac{1}{q'})\ge \max(\frac{1}{pq},\frac{1}{2p},\frac{1}{2q}), 
 \qquad (p,q)=(1,1).$$
\end{enumerate}

\end{remark}
\begin{proof}[Proof of item (1)]
First note that as a consequence of the Hölder inequality we have $$\norm{\widehat{f} \mathbbm{1}_{|\xi|\geq 1/2}}_1 \lesssim   \norm{\widehat{f}|\xi|^{\beta}}_q.$$ Second, defining $\Delta^1 (g) (x)=g(x)-g(x+1)$ and $\Delta^ j=\Delta( \Delta^{j-1})$, we have
$$|\Delta^k (\widehat{f}\,)(\xi)|\leq \norm{f (1-e^{2 \pi i x})^k}_1 \lesssim \norm{f |x|^{\alpha}}_p, $$where $k= \lceil \alpha \rceil$. Thus, 
$$
\norm{{f}}_\infty \leq
\norm{\widehat{f}}_1 \lesssim \norm{\widehat{f} \mathbbm{1}_{|\xi|\geq 1/2}}_1 + \norm{\Delta^k(\widehat{f}\,) \mathbbm{1}_{|\xi|\leq 1/2}}_1
\lesssim  \norm{f |x|^{\alpha}}_p + \norm{\widehat{f}|\xi|^{\beta}}_q.$$
Similarly, we obtain the same estimate for 
$
\norm{\widehat{f}}_\infty$,
whence the result follows.
\end{proof}

\section{Auxiliary lemma}
In this section we obtain a lemma which provides us with useful estimates for the norms of smooth step functions and their Fourier transforms. %In \cite{Kahane1994} it is implicitly used in the case $\Delta_k=A^k$ with $A>1$.

For a sequence $(c_j)$, we denote
 $$S_k(\xi)=\sum_{j=0}^k c_j e^{-2 \pi i j \xi}\quad\mbox{ and} \quad\widetilde{S_k}(\xi)= \sum_{j=k} ^N c_j e^{-2 \pi i j \xi}.$$

\begin{lemma}
\label{lemma:sbp}
Let $1\leq p,q\leq \infty$. Let $\widehat{\phi}$ be a Schwartz, even, non-increasing function which is equal to $1$ on $[-1/2,1/2]$ and supported on $(-1,1)$. Let $(\Delta_k)_{k=-1}^\infty$ be an increasing sequence with $\Delta_{-1}=0<1\leq \Delta_{0}$ and $\Delta_k \approx \Delta_{k+1} $ for $k\geq 0$. For any integer $N\geq 0$ and any sequence $(c_k)_{k=0}^N$, 
we set 
\begin{eqnarray}F(x)=\sum_{k=0}^N c_k \Delta_k \phi\left((x-k)\Delta_k\right)\end{eqnarray} and \begin{eqnarray} G(x)=  \sum_{k=0}^N c_k\left(\delta(x-k) -\Delta_k \phi( (x- k)\Delta_k)\right).
\end{eqnarray}
Then,  for any $\alpha, \beta\in \mathbb{R}$ and $\nu >\frac{1}{q'}$,
we have
\begin{enumerate}
    \item $ \displaystyle \norm{F (|x|+1)^\alpha}_{p} \lesssim \left(\sum_{k=0}^N |c_k|^{p} (k+1)^{p\alpha}\Delta_k^{p-1} \right)^{\frac{1}{p}}; $

   \begin{comment} \item 
    \begin{eqnarray*} \label{eq:normF3}\norm{F |x|^{-\alpha}}_{p'} \lesssim \left(\sum_{k=0}^N |c_k|^{p'} (1+k)^{-p'\alpha}\Delta_k^{p'-1} \right)^{\frac{1}{p'}}; \end{eqnarray*}
    \end{comment}
\item    $\displaystyle
    \norm{\widehat{F}(|\xi|+1)^\beta}_q \lesssim \left( \sum_{k=0}^N \Delta_k^{ \beta q} (\Delta_{k}-\Delta_{k-1}) \norm{\widetilde{S_k}}_{L^q(\mathbb{T})}^q \right)^{\frac{1}{q}};
$

\item $\displaystyle    \norm{\widehat{G}(|\xi|+1)^{-\nu }}_{q'} \lesssim \left( \sum_{k=0}^{N-1} \Delta_k^{ -\nu  q'} (\Delta_{k+1}-\Delta_{k}) \norm{{S}_k}_{L^{q'}(\mathbb{T})}^{q'} \right)^{\frac{1}{q'}}$\\
\phantom{x}\hspace{26mm}
%\qquad\qquad\qquad\qquad
$+ \Delta_N^{-\nu +\frac{1}{q'}} \norm{{S}_N}_{L^{q'}(\mathbb{T})}.$
\item  for any $M\in \mathbb{N}$ there exists a constant $K_M$ such that $$|F(n)-\phi(0)\Delta_n c_n|\leq K_M \norm{c}_\infty (n^{-1} + \Delta_{\lfloor n/2\rfloor}^{-1})^M$$
for any integer $n.$

\end{enumerate}

\end{lemma}

\begin{proof}
   To prove the first item, observe that, by Hölder's inequality, $$|F(x)| \leq  \left(\sum_{k=0}^N |c_k|^{p} \Delta_k^{p} |\phi( (x- k)\Delta_k)|\right)^{\frac{1}{p}} \left(\sum_{k=0}^N  |\phi( (x- k)\Delta_k)|\right)^{\frac{1}{p'}}.$$  The rapid decay of $\phi$ and the assumption $\Delta_k \geq 1$ imply that, for any $x,$ $$\sum_{k=0}^\infty  |\phi( (x- k)\Delta_k)|\lesssim 1.$$ Hence,
   $$\norm{F (|x|+1)^\alpha}_{p} \lesssim \left(\sum_{k=0}^N |c_k|^{p} \Delta_k^{p} \big\|{(|x|+1)^\alpha \phi( (x- k)\Delta_k)}\big\|_p ^p\right)^{\frac{1}{p}}.$$ Finally, once again because of the rapid decay of $\phi$, we conclude the proof of the first item by noting that $$\Delta_k^{\frac{1}{p}}\norm{(|x|+1)^\alpha \phi( (x- k)\Delta_k)}_p\lesssim (k + 1)^\alpha.$$

   To prove the second item, using the Abel transformation, we obtain that
   \begin{eqnarray*}
       \widehat{F}(\xi) &=& \sum_{k=0}^{N} c_k e^{-2 \pi i k \xi} \widehat{\phi}(\xi /\Delta_k)\\
       &=&\widetilde{S_0}(\xi)\widehat{\phi}(\xi/\Delta_0)+ \sum_{k=1}^N \widetilde{S_k} (\xi)( \widehat{\phi}(\xi/\Delta_k) - \widehat{\phi}(\xi/\Delta_{k-1}))\\
       &\leq & |\widetilde{S_0}(\xi)|\widehat{\phi}(\xi/\Delta_0)+ \left(\sum_{k=1}^N |\widetilde{S_k} (\xi)|^q |\widehat{\phi}(\xi/\Delta_k) - \widehat{\phi}(\xi/\Delta_{k-1}))|\right)^\frac{1}{q},
   \end{eqnarray*}
   where in the last inequality we used Hölder's inequality and the monotonicity of $\widehat{\phi}$. Next, because of the periodicity of $\widetilde{S_k}$, we have
   \begin{multline*}
       \norm{\widehat{F} (|\xi|+1)^\beta }_{q}\lesssim \Delta_0^{\beta+\frac1q} \norm{\widetilde{S_0}}_{L^q(\mathbb{T})} \quad\\ \qquad\quad+ \left(
   \sum_{k=1}^{N}  \norm{\widetilde{S_k}}^q _{L^q(\mathbb{T})}  \sup_{0 \leq \xi \leq 1}\left(\sum_{j\in \mathbb{Z}}   |\xi+j| ^{ \beta q}  \left| \widehat{\phi}\left(\frac{\xi+j}{\Delta_k}\right) - \widehat{\phi}\left(\frac{\xi+j}{\Delta_{k-1}}\right)\right |  \right) \right)^{\frac1q}.
   \end{multline*}
   Finally, by the support assumption, the inequality $|\widehat{\phi}(x)-\widehat{\phi}(y)|\lesssim |x-y|$ and the property $\Delta_k \approx \Delta_{k+1},$ we conclude that for $k\geq 1$
\begin{align*}
   &\sum_{j\in \mathbb{Z}}   |\xi+j| ^{ \beta q}  \left| \widehat{\phi}\left(\frac{\xi+j}{\Delta_k}\right) - \widehat{\phi}\left(\frac{\xi+j}{\Delta_{k-1}}\right)\right | \\&\lesssim\sum_{\Delta_{k-1}/2\leq j +\xi \leq \Delta_k} (j+\xi)^{\beta q+1} (\Delta_{k-1}^{-1} -\Delta_{k}^{-1}) \\&\lesssim \Delta_k^{q \beta +2 }(\Delta_{k-1}^{-1} -\Delta_{k}^{-1}) \\ &\approx  \Delta_k^{q \beta  }(\Delta_{k} -\Delta_{k-1}).
    \end{align*}
  The third item follows in a similar way. Indeed, by  the Abel transformation,
   
   \begin{align*}
       \widehat{G}(\xi) &= \sum_{k=0}^{N} c_k e^{-2 \pi i k \xi} (1-\widehat{\phi}(\xi /\Delta_k))\\&= {S}_N(\xi)(1-\widehat{\phi}(\xi/\Delta_N))+ \sum_{k=0}^{N-1} {S}_k (\xi)( \widehat{\phi}(\xi/\Delta_{k+1}) - \widehat{\phi}(\xi/\Delta_{k}))\\
       &\leq  |{S}_N(\xi)|(1-\widehat{\phi}(\xi/\Delta_N))+ \left(\sum_{k=0}^{N-1} |{S}_k (\xi)|^{q'} |\widehat{\phi}(\xi/\Delta_{k+1}) - \widehat{\phi}(\xi/\Delta_{k})|\right)^\frac{1}{q'}.
   \end{align*}
% where we used Hölder's inequality and the monotonicity of $\widehat{\phi}$.
Using the periodicity of ${S}_k$,
 \begin{multline}
 \label{eq:Kgaug}
      \norm{\widehat{G} (|\xi|+1)^{-\nu } }_{q'}\lesssim \Delta_N^{-\nu +\frac{1}{q'}} \norm{{S}_N}_{L^{q'}(\mathbb{T})} 
     \\ + \left(
   \sum_{k=0}^{N-1}  \norm{{S}_k}^{q'} _{L^{q'}(\mathbb{T})}  \sup_{0 \leq \xi \leq 1}\left(\sum_{j\in \mathbb{Z}}   |\xi+j| ^{ -\nu  q'}  \left| \widehat{\phi}\left(\frac{\xi+j}{\Delta_{k+1}}\right) - \widehat{\phi}\left(\frac{\xi+j}{\Delta_{k}}\right)\right |  \right) \right)^{\frac{1}{q'}}.
    \end{multline} 
   Finally, repeating the previous arguments we conclude that
\begin{align*}
&
   \sum_{j\in \mathbb{Z}}   |\xi+j| ^{ -\nu  q'}  \left| \widehat{\phi}\left(\frac{\xi+j}{\Delta_{k+1}}\right) - \widehat{\phi}\left(\frac{\xi+j}{\Delta_{k}}\right)\right |\\ &\lesssim \sum_{\Delta_{k}/2\leq \xi +j \leq \Delta_{k+1}} (\xi +j)^{-\nu  q'+1} (\Delta_{k}^{-1} -\Delta_{k+1}^{-1}) \\&\lesssim \Delta_k^{-q' \nu  +2 }(\Delta_{k}^{-1} -\Delta_{k+1}^{-1}) \approx  \Delta_k^{-q' \nu   }(\Delta_{k+1} -\Delta_{k}).
    \end{align*}

    To prove the fourth item, we use  the estimate $|\phi(x)|\lesssim (1+|x|)^{-M}$ to  deduce that

    \begin{align*}
    |F(n)&-c_n\Delta_n\phi(0)|\leq \sum_{k\neq n} \Delta_k |c_k \phi(\Delta_k (n-k))| \lesssim \norm{c}_\infty  \sum_{k\neq n} \frac{\Delta_k}{(\Delta_k|k-n|)^{M}}\\
    &
    \lesssim
    \norm{c}_\infty\left(\sum_{k=0}^{\frac{n}{2}} \frac{\Delta_k}{(\Delta_k n )^{M}} + \Delta_{\lfloor n/2\rfloor} ^{1-M} \sum_{k\neq n} \frac{1}{|k-n|^{M}} +  \sum_{k=2n}^{\infty} \frac{\Delta_k}{(\Delta_k k)^{M}}\right)\\
   &\lesssim \norm{c}_\infty\left( n^{1-M} +\Delta_{\lfloor n/2 \rfloor}^{1-M}\right).
    \end{align*}
\end{proof}

\section{Proofs of main results}
\subsection{Proof of Theorem \ref{theorem:mainth}}
It is a standard approach to establish the convergence of $\lim_N P_N(f)$ for any function satisfying $\norm{f |x|^{\alpha}}_p + \norm{\widehat{f}|\xi|^{\beta}}_q<\infty$ through
the uniform boundedness of $ P_N(f)$ with respect to
 %inequalities for $ P_N(f)$ which are uniform in 
 $N$. The following lemma gives sufficient conditions under which $P_N$ and its modifications are uniformly bounded: % The next lemma is thus natural.
\begin{lemma}
\label{lemma:Ngamma}
Let
$1\le p,q<\infty$  and  $\alpha - 1/p' ,\beta- 1/q'>0$. For $\gamma\ge 0$, we set $$Q_N(f)=\frac{1}{N^\gamma}\sum_{k=0}^N  c_{k} f(k)$$ with  a positive non-decreasing sequence $(c_k)_{k=0}^\infty$ satisfying  $c_k\lesssim k^\gamma$.

 Then the inequality 
\begin{equation}
    \label{eq:QN}
    |Q_N(f)|\lesssim \norm{f |x|^{\alpha}}_p + \norm{\widehat{f}|\xi|^{\beta}}_q
\end{equation}
holds uniformly on $N$ in the following cases:
\begin{enumerate}

     \item  if $(\alpha - 1/p') (\beta- 1/q')=1/(pq)$ and $\gamma>0$;
      \item if $p=q=1$, $\alpha \beta =1 $ and $\gamma \geq  0$.
\end{enumerate}

\end{lemma}

\begin{proof}

    Let $\widehat{\phi}$ be a smooth, non-increasing function supported on $(-1,1)$ and equal to $1$ on $[-1/2,1/2]$. Set
    $$\Delta_k = 1+ k^B, \qquad B=\frac{1/q}{\beta- 1/q'}.$$

   In the notation of Lemma \ref{lemma:sbp} and
    by Plancherel and Hölder's inequalities, we have
    \begin{align*}
    Q_N(f)&=N^{-\gamma}\left( \int_{\mathbb{R}} f \overline{F} + \int_{\mathbb{R}} \widehat{f} \overline{\widehat{G}}\right) \\
    &\lesssim N^{-\gamma} \left(\norm{f |x|^{\alpha}}_p + \norm{\widehat{f}|\xi|^{\beta}}_q\right) \left(\norm{F (1+|x|)^{-\alpha}}_{p'}+ \norm{\widehat{G}(1+|\xi|)^{-\beta}}_{q'}\right) ,\end{align*} where in the last estimate we also used item (1) of Remark \ref{remark}. Thus, the result follows if we prove that 
    $$\norm{F (1+|x|)^{-\alpha}}_{p'}+ \norm{\widehat{G}(1+|\xi|)^{-\beta}}_{q'} \lesssim N^{\gamma}.$$
To estimate $\norm{F (1+|x|)^{-\alpha}}_{p'}$ we use item (1) in Lemma \ref{lemma:sbp}. Since $c_k\lesssim k^\gamma$ and $(\alpha - 1/p') (\beta- 1/q')= 1/(pq)$, we have
      \begin{eqnarray*}
      \norm{F (1+|x|)^{-\alpha}}_{p'} \lesssim N^\gamma.
  \end{eqnarray*}
 Next we deal with the quantity $\norm{\widehat{G}(1+|\xi|)^{-\beta}}_{q'}$ for $q>1$. Since the sequence  $(c_k)_k$ is non-decreasing, we can apply the Hardy-Littlewood inequality for monotone Fourier coefficients (see Chapter XII in \cite{zygmund}) to obtain that  $$\int_0 ^1 |{S}_k|^{q'} \lesssim \sum_{j=1}^k c_j^{q'} (k-j)^{q'-2} \lesssim   k^{q' \gamma +q'-1}.$$ If $q=1$, it is clear that
 $\norm{S_k}_\infty \lesssim k^{\gamma +1}$. Therefore, for any $q\geq 1$ by using item (3) in Lemma \ref{lemma:sbp} we deduce that
   $$\norm{\widehat{G}(1+|\xi|)^{-\beta}}_{q'} \lesssim N^{\gamma}.$$
   
   The proof is now complete.
 \end{proof}

\begin{remark}\label{remark1}
   If $(\alpha - 1/p') (\beta- 1/q')>1/(pq)$ and $\gamma = 0$, the same proof with $\frac{\frac{1}{q}}{\beta - \frac{1}{q'}}<B<\frac{\alpha - \frac{1}{p'}}{\frac{1}{p}}$ shows that \eqref{eq:QN} holds true. This yields  an alternative proof of item (1) in Theorems \ref{theorem:kah} and \ref{theorem:mainth}.
\end{remark}
We are now in a position to prove our main results.

\begin{proof}[Proof of Theorem \ref{theorem:mainth}]
 The first and the forth  statements are contained in Theorem \ref{theorem:kah}, see also  Remark 
 \ref{remark1} for an alternative proof of the first one. The second item follows from item (2) in Lemma \ref{lemma:Ngamma} with $\gamma=0$ by a standard density argument (see  also   Remark \ref{remark} (3)). We now prove the positive part of the third item.

    Let $\widehat{\phi}$ be a smooth, even, non-increasing function supported on $(-1,1)$ such that $\widehat{\phi}(\xi)\equiv 1$ for $|\xi|\leq 1/2$. Set
 $M = \lceil N^{\frac{ \beta- 1/q'}{1/q}} \rceil = \lceil N^{\frac{1/p}{\alpha - 1/p'}}\rceil $ and let $$P_{\Phi,N}(\widehat{f}):=\sum_{k\in \mathbb{Z}} \widehat{f}(k) \widehat{\phi}(k/N).$$
We show that both \begin{equation}
\label{eq:regu} \left |P_{\Phi,N}(\widehat{f}) -P_M(f) \right| \lesssim \norm{f |x|^{\alpha}}_p + \norm{\widehat{f} |\xi|^{\beta}}_q 
\end{equation}
and \begin{equation}
\label{eq:regu2} \left |P_{\Phi,N}(\widehat{f}) -P_N(\widehat{f}) \right| \lesssim \norm{f |x|^{\alpha}}_p + \norm{\widehat{f} |\xi|^{\beta}}_q \end{equation} hold independently of $N$. By the density of the Schwartz functions, for which \eqref{eqn:einstein}
 clearly holds,
 one has
 \begin{eqnarray*}
\label{eq:phi}
    \lim _N \big(P_{N}(\widehat{f}\,) -P_M({f}) \big)=0
\end{eqnarray*}
for any $f$ satisfying \eqref{space}.

First, to prove inequality \eqref{eq:regu}, we observe that by Parseval and Hölder's inequality, % we have

\begin{align*}
    \left |P_M(f)-\int_{\mathbb{R}} f(x) \sum_{|k|\leq M} N \phi\left( (x-k)N\right) dx\right| &=
\left |\int_{\mathbb{R}} \widehat{f}(\xi) D_M(\xi)\left(1-  \widehat{\phi}(\xi/N) \right) d\xi \right| \\
&\lesssim\norm{\widehat{f} \xi ^{\beta}}_{q}  M^{\frac{1}{q}}N^{-\beta+ \frac{1}{q'}} \approx \norm{\widehat{f} \xi ^{\beta}}_{q} .
\end{align*} Here for 
the Dirichlet kernel 
$D_M(\xi)=\sum_{|k|\leq M}e^{2 \pi i k \xi}$ we have used the estimate
$$ \left(\int_{N/2}^\infty |D_M(\xi)|^{q'} |\xi|^{-q' \beta} d \xi \right)^{\frac{1}{q'}} \lesssim M^{\frac{1}{q}} N ^{-\beta+ \frac{1}{q'}}.$$
A further application of Hölder's inequality shows that
$$ \left |\int_\mathbb{R} f(x) \sum_{|k|\geq M} N \phi( (x-k)N) dx \right| \lesssim  \norm{f (1+|x|)^{\alpha}}_{p} N^{\frac{1}{p}} M^{-\alpha + \frac1{p'}} \approx \norm{f (1+|x|)^{\alpha}}_{p}.$$
Hence, by the Poisson summation formula 
$$\sum_{k\in \mathbb{Z}}  \widehat{\phi}(k/N)
e^{-2\pi i kx}=
\sum_{k\in \mathbb{Z}}   N \phi\left( (x-k)N\right),$$
we deduce that
\begin{align*}
    \left |P_M(f) -P_{\Phi,N}(\widehat{f}) \right| &=  \left |P_M(f)-\int_{\mathbb{R}} f(x) \sum_{k \in \mathbb{Z}} N \phi\left( (x-k)N\right) dx\right|\\&\lesssim \norm{f |x|^{\alpha}}_{p} + \norm{\widehat{f} |\xi| ^{\beta}}_{q}.
    \end{align*}
This proves inequality \eqref{eq:regu}.

Inequality \eqref{eq:regu2} follows by applying item (1) in Lemma \ref{lemma:Ngamma} with $\gamma=1$ to $$\sum_{|k|\leq N} \widehat{f}(k) (1-\widehat{\phi}(k/N)),$$ because $0 \leq 1-\widehat{\phi}(k/N)\lesssim k/ N$ and is $\widehat{\phi}$ is non-increasing.

We now prove the negative part of the third item.
Without loss of gene\-rality,
we can  assume that $p\geq 1$ and $q>1$. We show that there exists a function $f$ such that \begin{equation}
       \label{infty}
   \limsup P_N(f)= - \liminf P_N(f)= + \infty.
\end{equation} Since we know that $\lim_N (P_N(f) - P_{N^\gamma}(\widehat{f}\,))=0$
with some 
 $\gamma(\alpha,\beta,p,q)>0$, the same divergence result holds for $P_N(\widehat{f})$.

   Let $N>0$. We  define \begin{eqnarray*}F_N(x)=\sum_{k=0}^N c_k \Delta_k \phi\left(\Delta_k(x-k)\right),\end{eqnarray*} where $$c_k= \frac{1}{ (k+1) ^{1+B}
   \log(k+2)^A }$$ and 
   $$\Delta_k= (k+1)^B \log(k+2)^{A-1},$$ with $A= \frac{\beta}{\beta- 1/q'}>1$ and $B=\frac{ \alpha - 1/p'}{1/p}= \frac{1/q}{ \beta - 1/q'}$.
   By using item (4) in Lemma \ref{lemma:sbp}, noting that
   $$\sum_n
   (n^{-1} + \Delta_{\lfloor n/2 \rfloor}^{-1})^{M'}\approx 1$$
   with large enough $M'$,
   we then have, for $M,N$ large enough, 
   \begin{equation}\label{vsp}
   P_M(F_N) \approx \sum_{n=0}^{\min (N,M)} c_n \Delta _n \approx \min(\log \log N,\log \log M).
   \end{equation}

   Besides, Lemma \ref{lemma:sbp} implies both 
   \begin{equation}\label{vsp1}\norm{ F_N |x|^ \alpha}_{p} \lesssim  1 \end{equation} and     \begin{equation}\label{vsp2}\norm{\widehat{F_N} |\xi|^\beta }_{q} \lesssim  \log ^\frac 1q \log N, \end{equation}
   where, in the last inequality, we also used the Hardy-Littlewood estimate 
   $$\int_0 ^1 |\widetilde{S_k}|^{q} \approx \sum_{j=k}^N c_j^{q} (j-k+1)^{q-2} \lesssim   k^{-Bq -1 } \log^{-qA}(k+2).$$

   Finally, for $(N_k)_{k=1}^\infty$ a rapidly increasing sequence, let 
   $$f(x)=\sum_{j=1}^\infty \frac{(-1)^j}{j^2} \frac{F_{N_j}(x)}{\log^{\frac{1}{q}} \log N_j}.$$ 
   Using the estimates of $\norm{ F_N |x|^ \alpha}_{p}$ and $\norm{ \widehat{F_N} |\xi|^\beta}_{q}$,  we arrive at
    $\norm{ 
    f |x|^ \alpha}_{p} + \norm{ \widehat{f} |\xi|^\beta }_{q}<\infty$. 

We now show that 
   \begin{equation}\label{vsp3}
\lim _{k \to \infty } (-1)^k P_{N_k}(f)= + \infty
\end{equation} provided that $N_k$ increases rapidly enough. 
    Indeed, %for even $k,$
$$P_{N_k}(f)=\left(\sum_{j=1}^k
+
\sum_{j=k+1}^\infty
\right)\frac{(-1)^j}{j^2} \frac{P_{N_k}(F_{N_j}(x))}{\log^{\frac{1}{q}} \log N_j} =I_1+I_2.
$$
Taking into account \eqref{vsp}, we derive
\begin{align*}
  (-1)^k  I_1&\gtrsim
\frac{1}{k^2} \frac{P_{N_k}(F_{N_k}(x))}{\log^{\frac{1}{q}} \log N_k}-
\sum_{j=1}^{k-1}
\frac{1}{j^2} \frac{P_{N_k}(F_{N_j}(x))}{\log^{\frac{1}{q}} \log N_j}
\\&\gtrsim
\frac{1}{k^2} {\log^{1-\frac{1}{q}} \log N_k}-
{\log^{1-\frac{1}{q}} \log N_{k-1}}\to\infty.
\end{align*}
%Using the fact that $|P_N(g)|\lesssim N\|\widehat{g}\|_1
%\lesssim N(
%\norm{g}_{p,\alpha} + \norm{\widehat{g}}_{q, \beta})
%$ (see Remark \ref{remark})
%and estimates \eqref{vsp1}, %\eqref{vsp2},
%we also have
Using again \eqref{vsp},
\begin{align*}
    |I_2|&\lesssim \sum_{j=k+1}^{\infty}
\frac{1}{j^2} \frac{|P_{N_k}(F_{N_j}(x))|}{\log^{\frac{1}{q}} \log N_j}
\\&\lesssim
\sum_{j=k+1}^{\infty}
\frac{1}{j^2} \frac{ \log \log N_k}{\log^{\frac{1}{q}} \log N_j}\lesssim
\frac{ \log \log N_k}{\log^{\frac{1}{q}} \log N_{k+1}}
\to0.
\end{align*}

   We have thus proved \eqref{vsp3} and, therefore, \eqref{infty}.

   %If $q=1$!!??, $\int_0 ^1 |\widetilde{S_k}| \lesssim c_k.$ Hence
   
   %$$\norm{\widehat{F} }_{1,\beta} \lesssim  \sum_{k} c_k k^{B \beta + B -1}$$

\end{proof}
\subsection{Proof of Theorem \ref{theorem:extkah2}}

%In their construction of counterexamples for the absolute convergence of \eqref{eqn:einstein}, in \cite{Kahane1994} the authors use the explicit Rudin-Shapiro construction of polynomials with comparable $L^2$ and $L^p$ norms. 

The following lemma is based on the classical Rademacher-Salem-Zygmund  randomization technique: 
%which allows us to construct many instances of such polynomials: 
\begin{lemma}
\label{lemma:random}
For any sequences $(w_k)_{k=0}^N$ and $(c_k)_{k=0}^N$ with $w_k\geq 0$, there exists a choice of signs $(\varepsilon_k)_{k=0}^N$ such that, setting $S_n=\sum_{k=0}^{n} c_k \varepsilon_k e^{2 \pi i k x}$,
\begin{eqnarray}
     \left( \sum_{k=0}^N w_k \norm{{S}_k}_{L^q(\mathbb{T})}^q \right)^{\frac{1}{q}} \lesssim% (\log N)^{q} 
     \left( \sum_{k=0}^N w_k \norm{{S}_k}_{L^2(\mathbb{T})}^q \right)^{\frac{1}{q}}, \quad q<\infty,
\end{eqnarray}
and, for $q=\infty$,
$$    \norm{{S}_k}_{L^\infty(\mathbb{T})} \lesssim (\log (k+1))^{\frac{1}{2}}  \norm{{S}_k}_{L^2(\mathbb{T})}, \quad 0\leq k\leq N.
$$
%where $q ^\# =0$ for $0<q<\infty$ and $q ^\# =\frac{1}{2}$ for $q=\infty$.

\end{lemma}
\begin{proof}
If $q<\infty$, by Khintchin's inequality (see for instance Theorem 6.2.3 in \cite{albiac2006topics}), taking the average over all possible sign combinations we deduce that
\begin{eqnarray*}
     \mathbb{E}\left( \sum_{k=0}^N w_k \norm{{S}_k}_{L^q(\mathbb{T})}^q \right) \lesssim \left( \sum_{k=0}^N w_k \norm{{S}_k}_{L^2(\mathbb{T})}^q \right),
\end{eqnarray*}
whence the result follows. If $q=\infty$, it is shown in   \cite[p. 271]{SalemZygmund}  that there exists $c>0$ such that
    $$\sum_{n=1}^\infty \exp\left({ \frac{\log^\frac12(n+1)  \norm{S_n}_\infty}{  \norm{S_n}_2} - c \log (n+1)} \right)\lesssim 1$$ for some choice of signs, whence the result follows at once.
\end{proof}

\begin{proof}[Proof of Theorem \ref{theorem:extkah2}]

\begin{comment}
    The first  item follows from the inequality (and the analogous one with $\widehat{f}$ and $f$ permuted)

$$\|\widehat{f}\|_1\le\norm{f |x|^{\alpha}}_p + \norm{\widehat{f}|\xi|^{\beta}}_q,
$$which can be obtained by noting that
$$\norm{\Delta_1^k \widehat{f}}_ \infty \lesssim \norm{f |x|^{\alpha}}_p$$ for $k \geq \alpha$ and $\alpha> \frac{1}{p'}$ and 
$$\norm{\widehat{f} \mathbbm{1}_{|\xi\geq h}} \lesssim _h \norm{\widehat{f} |\xi|^\beta},$$ for any $h>0$ and $\beta>\frac{1}{q'}$.

\end{comment}

In view of Theorems \ref{theorem:kah2} and \ref{theorem:mainth}, it remains to consider only three cases:
\begin{itemize}

\item[$\cdot$] $q\leq2$ and 
$(\alpha - \frac{1}{p'})(\beta- \frac{1}{q'})\leq\frac{1}{pq}, \quad (p,q)\neq(1,1),$

\item[$\cdot$]
    $p=q=1$, $(\alpha - \frac{1}{p'})(\beta- \frac{1}{q'})=\frac{1}{pq}$,
    
\item[$\cdot$]$q>2$, $(\alpha - \frac{1}{p'})(\beta- \frac{1}{q'})=\frac{1}{2p}$.
   
\end{itemize}

First, let $q\leq2$ and $(p,q)\neq(1,1)$. By item (1)
in Remark \ref{remark}, it suffices to consider the case 
$(\alpha - \frac{1}{p'})(\beta- \frac{1}{q'})=\frac{1}{pq}$.
In this case, by Theorem \ref{theorem:mainth}  (3),
there exists 
a function $f$ such that $\lim_N P_N(f)$ does not exist, which implies that $\lim_N P_N(|f|)=\infty.$

Second,
if $p=q=1$ and $a \beta=1$, set $\Delta_n=n^\alpha$.
By the triangle inequality,
\begin{align*}
    \sum_{n=1}^\infty |f(n)|&\leq \sum_{n=1}^\infty  \left| f(n)- \int_\mathbb{R} \Delta_n \phi(\Delta_n(x-n)) f(x) dx \right|\\ &+
    \sum_{n=1}^\infty \left|\int_\mathbb{R} \Delta_n \phi(\Delta_n(x-n)) f(x) dx \right|,
\end{align*} where, as usual, $\widehat{\phi}$ is a smooth, even, non-increasing function supported on $(-1,1)$ such that $\widehat{\phi}(\xi)\equiv 1$ for $|\xi|\leq 1/2$.
It is plain to see (for instance, by Lemma \ref{lemma:sbp} item (4)) that  one can estimate the second term as follows:
$$\int_\mathbb{R}|f(x) |
\sum_{n=1}^\infty \left|\Delta_n \phi(\Delta_n(x-n)) \right|dx \lesssim \norm{f(x)(1+|x|)^\alpha}_1.$$
For the first one, by Plancherel's formula,
\begin{align*} \sum_{n=1}^\infty \left| f(n)- \int_\mathbb{R} \Delta_n \phi(\Delta_n(x-n)) f(x) dx \right|&\leq \int_\mathbb{R}  |\widehat{f}(\xi)| \sum_n |1- \widehat{\phi}(\xi/\Delta_n)| d \xi \\
&\lesssim \norm{\widehat{f} |\xi|^\beta}_1,
\end{align*} where we used that, for $M$ large enough, 
 $$\sum_{n=1}^\infty  |1- \widehat{\phi}(\xi/\Delta_n)| \lesssim \sum_{n=1}^\infty \min\Big(1, \Big(\frac{\xi}{n^\alpha}\Big)^M \Big) \lesssim |\xi|^{1/\alpha}=|\xi|^{\beta}.$$

Let now $q>2$. The counterexample is similar to the one used in  Theorem \ref{theorem:mainth}. Assume first that $q<\infty$. Set $$c_k= \varepsilon_k \frac{1}{  (1+ k) ^{1+B}
\log(k+2)^A},$$
%$\varepsilon_k=\pm1$, 
and $$\Delta_k= (1+k)^B \log(k+2)^{A-1},$$ with $A= \frac{\beta}{\beta- 1/q'}$ and
$B=\frac{ \alpha - 1/p'}{1/p}= \frac{1/2}{ \beta - 1/q'}$.

%The proof proceeds similarly as in  Theorem \ref{theorem:mainth}. 
For $N\in \mathbb{N}$ define
$F_N(x)=\sum_{k=0}^N c_k \Delta_k \phi\left(\Delta_k(x-k)\right)$. Then
we obtain (compare with \eqref{vsp} and \eqref{vsp1})
$$  \lim_M P_M(|F_N|) \approx \sum_{n=0}^N |c_n| \Delta _n \approx \log \log N
$$
and $$\norm{ F_N |x|^ \alpha}_{p} \lesssim  1.$$
To derive the estimate of $\norm{\widehat{F_N}|\xi|^\beta}_q$, cf.
\eqref{vsp2},
we 
observe that by Lemma \ref{lemma:random} there exists a choice of signs $\varepsilon_k$ such that in item (2) of Lemma \ref{lemma:sbp} the $L^q$ norms can be replaced by $L^2$ norms,
namely,
$$ \norm{\widehat{F_N}|\xi|^\beta}_q \lesssim \left( \sum_{k=0}^N \Delta_k^{ \beta q} (\Delta_{k}-\Delta_{k-1}) \norm{\widetilde{S_k}}_{L^2(\mathbb{T})}^q \right)^{\frac{1}{q}}.
$$
Since that 
$$\norm{\widetilde{{S}_k}}_{L^2(\mathbb{T})}^q=\left(\sum_{j=k}^N |c_j|^2 \right)^{\frac{q}{2}} \lesssim k^{-q(1+B)+\frac{q}{2}} \log^{-qA}(k+2),$$ we deduce that
$\norm{\widehat{F_N}|\xi|^\beta}_q\lesssim \log ^{\frac{1}{q}}\log N $.

Thus, in the case $q>2$,
the inequality 
\begin{equation}\label{vsp10}
\sum_n|f(n)|\lesssim \norm{ f|x|^ \alpha}_{p} 
+
\norm{\widehat{f}|\xi|^\beta}_q
\end{equation}
does not hold. Therefore,
there is a function $f$
satisfying \eqref{space} for which 
$\lim_N P_N(|f|)=\infty$.

\begin{comment}
To deal with $q=\infty$, first, if $q=\infty$ and $p>1$, 
set $$c_k= \varepsilon_k \frac{1}{  (1+ k) ^{1+B}}$$ and $$\Delta_k= (1+k)^B.$$ Then we estimate as usual
 $\sum_{n=0}^N |c_n| \Delta _n \approx  \log N$
and $\norm{F |x|^\alpha }_{p} \lesssim \log(N+1)^{\frac 1p}.$
Using Lemma \ref{lemma:random} there exists a choice of signs $\varepsilon_k$ such that in item (4) the $L^q$ can be replaced by $L^2$ norms up to $\log (N+1)^\frac 12$.
Thus
 $$\norm{ \xi^ \beta \widehat{F} }_{\infty} \lesssim  \log (N+1)^\frac 12.$$
\end{comment}

Finally, we deal with the case $q=\infty$, for which we recall that $\beta>1$ and
$ (\alpha - 1/p')(\beta - 1)= 1/2p$.
Set $$c_k= \begin{cases}
    \varepsilon_k \frac{1}{ (1+ k) ^{1+B}}, \text{ if }\sqrt N \leq k \leq N\\
    0, \text{ otherwise}
\end{cases}$$  and $$\Delta_k=1+ (1+k)^B (\log N)^{p^\#\frac{1}{4 (\beta -1)}},$$ where 
  $A= \frac{\beta}{\beta- 1}$ and
$B=\frac{ \alpha - 1/p'}{1/p}$, 
 %$= \frac{1/2}{ \beta - 1}$,
$p^\# = 0$ if $p>1$ and $p^\# = 1$ if $p=1$.  It is easy to see
that
$$ \lim_M P_M(|F_N|)\approx \sum_{n=0}^N |c_n| \Delta _n \approx 
%\log N 
(\log N)^{p^\#\frac{1}{4 (\beta -1)}+1}.
$$ 
Next, by item (1) Lemma  \ref{lemma:sbp},
$$\norm{F_N |x|^\alpha }_{p}\lesssim % (\log N)^{p^\#\frac{1}{4p' (\beta -1)}} 
\log^{\frac{1}{p}}N .$$

Finally, in the same way as before, by Lemma \ref{lemma:random} with $q=\infty$, we deduce that there exists a choice of signs for which for all $k$
\begin{align*}
    \norm{\widetilde{S_k}}_\infty 
&\lesssim 
\log^\frac{1}{2}(N+1) \left( \sum_{j=\max({k,\sqrt{N}})}^{N}\frac{1}{j^{2+2B}}
\right)^\frac{1}{2}
\\&\lesssim \log^\frac{1}{2}(N+1)  \min\left(N^{-1/4-B/2},k^{-1/2-B}\right),
\end{align*}
 which implies that
 $$\norm{ |\xi|^ \beta \widehat{F_N} }_{\infty} \lesssim (\log N)^{p^\#\frac{\beta}{4 (\beta -1)}} \log^\frac{1}{2}(N+1) .$$
Combining the above estimates, we see that inequality \eqref{vsp10}
 %$$P_N(|f|)\lesssim \norm{ F |x|^\alpha}_p +  \norm{ \widehat{F} |\xi|^\beta}_q$$ 
  does not hold. The proof is concluded.
\end{proof}
\section{%Further remarks
%Remark  on 
Poisson formula
for general weights
}
Lemma \ref{lemma:sbp} can be easily generalized in the case of general weights. 
\begin{lemma}
\label{lemma:sbp2}
Under all the conditions of Lemma \ref{lemma:sbp}, assuming that $w$ is an even weight function, we have
\begin{enumerate}
    \item $\displaystyle \norm{F w}_{p} \lesssim \left(\sum_{k=0}^N |c_k|^{p} \norm{w(x) \phi( (x- k)\Delta_k)}_p^{p}\Delta_k^{p} \right)^{\frac{1}{p}}; $

   \begin{comment} \item 
    \begin{eqnarray*} \label{eq:normF3}\norm{F |x|^{-\alpha}}_{p'} \lesssim \left(\sum_{k=0}^N |c_k|^{p'} (1+k)^{-p'\alpha}\Delta_k^{p'-1} \right)^{\frac{1}{p'}}; \end{eqnarray*}
    \end{comment}
\item  if $w$ non-decreasing,  
\begin{align*}%\displaystyle
 \norm{\widehat{F} w }_{q}&\lesssim \left(\int^{\Delta_0+2}_0 w ^{q}\right)^\frac{1}{q} \norm{\widetilde{S_0}}_{L^q(\mathbb{T})}\\& + \left( \sum_{k=1}^N w(\Delta_k)^{ q} (\Delta_{k}-\Delta_{k-1}) \norm{\widetilde{S_k}}_{L^q(\mathbb{T})}^q \right)^{\frac{1}{q}};
\end{align*}

\item if $w$ is non-increasing, \begin{align*}   \norm{\widehat{G}w}_{q'} &\lesssim \left( \sum_{k=0}^{N-1} w^{q'}(\Delta_k/2) (\Delta_{k+1}-\Delta_{k}) \norm{{S}_k}_{L^{q'}(\mathbb{T})}^{q'} \right)^{\frac{1}{q'}}\\&+ \left(\int^\infty_{\frac{\Delta_N}2-1} w ^{q'}\right)^\frac{1}{q'} \norm{{S}_N}_{L^{q'}(\mathbb{T})}.
\end{align*}

\end{enumerate}
\begin{comment}
In particular, for $\alpha >-\frac{1}{p}$, $\beta>\frac{1}{q}$ we have
\begin{enumerate}\addtocounter{enumi}{3}
    \item  $ \displaystyle \norm{F |x|^\alpha}_{p} \lesssim \left(\sum_{k=0}^N |c_k|^{p} (1+k)^{p\alpha}\Delta_k^{p-1} \right)^{\frac{1}{p}}; $

\item    $\displaystyle
    \norm{\widehat{F}|\xi|^\beta}_q \lesssim \left( \sum_{k=0}^N \Delta_k^{ \beta q} (\Delta_{k}-\Delta_{k-1}) \norm{\widetilde{S_k}}_{L^q(\mathbb{T})}^q \right)^{\frac{1}{q}};
$

\item $\displaystyle    \norm{\widehat{G}|\xi|^{-\beta}}_{q'} \lesssim \left( \sum_{k=0}^{N-1} \Delta_k^{ -\beta q'} (\Delta_{k+1}-\Delta_{k}) \norm{{S}_k}_{L^{q'}(\mathbb{T})}^{q'} \right)^{\frac{1}{q'}}+ \Delta_N^{-\beta+\frac{1}{q'}} \norm{{S}_N}_{L^{q'}(\mathbb{T})}.$

\end{enumerate}
\end{comment}

\end{lemma}
As a consequence, we obtain a sufficient condition for the validity of \eqref{eqn:einstein}
 in 
 Lebesgue spaces
with general weights. Recall that the Dirichlet kernel 
is given by  $D_n(x)=\sum_{k=-n}^n e^{2 \pi i kx }$.
\begin{theorem}\label{corollary} Let $u,v$ be even, non-decreasing weights and $1\leq p,q\leq  \infty$. Assume that the Schwartz functions are dense in the space with norm 
     $\norm{fu}_p + \norm{\widehat{f}v}_q$
  and  that  there exist two sequences $\Delta$ and $\widetilde{\Delta}$ as in Lemma \ref{lemma:sbp}
  such that
 \begin{equation*}
     \sup_N \left(\sum_{k=0}^N  \norm{u^{-1}(x) \phi( (x- k)\Delta_k)}_{p'}^{p'}\Delta_k^{p'} \right)^{\frac{1}{p'}}<\infty;
 \end{equation*}
 \begin{align*}
    \sup_N  \Bigg( \sum_{k=0}^{N-1} v^{-q'}(\Delta_k/2) (\Delta_{k+1}&-\Delta_{k}) \norm{{D}_k}_{L^{q'}(\mathbb{T})}^{q'} \Bigg)^{\frac{1}{q'}}\\
    &+ \left(\int^\infty_{\frac{\Delta_N}{2}-1} v^{-q'}\right)^\frac{1}{q'} \norm{{D}_N}_{L^{q'}(\mathbb{T})}<\infty;
 \end{align*}
 \begin{equation*}
     \sup_N \left(\sum_{k=0}^N  \norm{v^{-1}(x) \phi( (x- k)\widetilde{\Delta}_k)}_{q'}^{q'}\widetilde{\Delta}_k^{q'} \right)^{\frac{1}{q'}}<\infty;
 \end{equation*} and 
 \begin{align*}
   \sup_N  \Bigg( \sum_{k=0}^{N-1} u^{-p'}(\widetilde{\Delta}_k/2) (\widetilde{\Delta}_{k+1} &-\widetilde{\Delta}_{k}) \norm{{D}_k}_{L^{p'}(\mathbb{T})}^{p'} \Bigg)^{\frac{1}{p'}}\\&+ \left(\int^\infty_{{\frac{\widetilde{\Delta}_N}{2}-1}} u^{-p'}\right)^\frac{1}{p'} \norm{{D}_N}_{L^{p'}(\mathbb{T})}<\infty.
 \end{align*}
    Then, for any $f$ with $\norm{fu}_p + \norm{\widehat{f}v}_q<\infty$,  the formula \eqref{eqn:einstein}
 holds.  
\end{theorem}
\begin{remark}
    It is possible to obtain an analogue of Theorem \ref{corollary} for the absolute convergence of $P_N(|f|)$ by using the same ideas as in the proof of Theorem \ref{theorem:extkah2}, item (2).
\end{remark}

\bibliographystyle{amsplain}
\bibliography{main}

\begin{comment}
Kahane, J.-P., Lemarié-Rieusset, P.-G.: Remarques sur la formule sommatoire de Poisson. Stud. Math.
109, 303–316 (1994)

K. Gr¨ochenig, An uncertainty principle related to the Poisson summa-
tion formula, Studia Math., 121, (1996), 87-104.

Y. Katznelson, An Introduction to Harmonic Analysis, Wiley, New York,
1968.

Katznelson, Y.: Une remarque concernant la formule de Poisson. Stud. Math. 29, 107–108 (1967)

Nguyen, H.Q., Unser, M., Ward, J.P. Generalized Poisson Summation Formulas for Continuous Functions of Polynomial Growth. J Fourier Anal Appl 23, 442–461 (2017). 
\end{comment}

\end{document}